\documentclass[preprint]{elsarticle}

\usepackage{hyperref}

\usepackage{amsmath}
\usepackage{amsfonts}
\usepackage{amssymb}
\usepackage{graphicx}
\usepackage{comment}
\usepackage{subcaption}
\usepackage[ruled]{algorithm2e}
\usepackage{algorithmic}
\usepackage{doi,url}

\usepackage{exscale}

\usepackage{color}

\newtheorem{lemma}{Lemma}

\newtheorem{theorem}{Theorem}

\newtheorem{corollary}{Corollary}

\newcommand{\eq}[1]{\begin{equation}\label{#1}}
\newcommand{\en}{\end{equation}}

\newenvironment{proof}{\begin{trivlist}
                       \item[]{\bf Proof.}
                       \hspace{0cm} }{\hfill $\Box$
                       \end{trivlist}}

\def\IC{\mathbb{C}}

\def\calY{\mathcal{Y}} 
\def\calV{\mathcal{V}} 
\def\calU{\mathcal{U}} 
\def\calX{\mathcal{X}} 
\def\calK{\mathcal{K}} 
\def\calQ{\mathcal{Q}} 
\def\calX{\mathcal{X}} 
 
\def\calW{\mathcal{W}} 

\def\inv{^{-1}}

\begin{document}

\begin{frontmatter}

\title{A generalization of Saad's bound on harmonic Ritz vectors of Hermitian matrices
}

%
%
%
%

\author{Eugene Vecharynski\corref{mycorrespondingauthor}}
\cortext[mycorrespondingauthor]{Corresponding author}
\ead{eugene.vecharynski@gmail.com}
\ead[url]{http://evecharynski.com/}
\address{Computational Research Division, Lawrence Berkeley National Laboratory; 1 Cyclotron Road, Berkeley, CA 94720, USA}

%
%

\begin{abstract}
We prove a Saad's type bound for harmonic Ritz vectors of a Hermitian 
matrix. The new bound reveals a dependence of 
the harmonic Rayleigh--Ritz procedure on the condition number 
of a shifted problem operator. Several practical implications
are discussed. In particular, the bound motivates 
incorporation of
preconditioning into
the harmonic Rayleigh--Ritz scheme.  
\end{abstract}


\end{frontmatter}


\section{Introduction}\label{sec:intro}

The Rayleigh--Ritz procedure is a well known technique for approximating
eigenpairs $(\lambda,x)$ of an $n$-by-$n$ matrix $A$ over a given subspace~$\calK$~\cite{Parlett:98, Saad-book3, Stewart-eigen-book}. 
It produces~approximate eigenpairs $(\mu, u)$, called the Ritz pairs, that satisfy the Galerkin condition 
\[
A u - \mu u \perp \calK, \quad u \in \calK.
\]
This is done by solving an $s$-by-$s$ eigenvalue problem
\eq{eq:rr}
K^* A K c = \mu K^*K c,
\en
where $K$ is a matrix whose columns contain a basis of 
$\calK$ and $s = \dim(\calK)<<n$. 
The eigenvalues $\mu$ of the projected problem~\eqref{eq:rr}, called the Ritz values,
represent approximations to the eigenvalues $\lambda$ of $A$. The 
associated eigenvectors $x$ are approximated by the Ritz vectors $u = Kc$. 

For a Hermitian matrix $A$, a general a priori bound 
that describes the approximation quality of Ritz vectors 
is due to Saad~\cite[Theorem~4.6]{Saad-book3}. The bound shows that 
the proximity of a Ritz vector $u$ to an exact eigenvector $x$ 
is determined essentially by the angle    
between this eigenvector and the subspace~$\calK$, defined as  
\eq{eq:sin_def}
\angle (x, \calK) = \min_{y \in \calK, y\neq0} \angle (x,y).
\en
This result (in a slightly generalized form) is stated in Theorem~\ref{thm:saad}.

\begin{theorem}[Saad~\cite{Saad-book3}]\label{thm:saad}
Let $(\lambda,x)$ be an eigenpair of a Hermitian matrix $A$ and $(\mu, u)$ be a Ritz pair with respect to the
subspace~$\calK$. 
Assume that~$\Theta$ is a set of all the Ritz values and 
let $P_{\calK}$ be an orthogonal projector onto $\calK$. Then    
\eq{eq:saad}
\sin \angle (x, u) \leq \sqrt{1 + \frac{\gamma^2}{\delta^2}} \sin \angle(x, \cal K),
\en
where
$\gamma = \|P_{\calK} A (I - P_{\calK} ) \|$
and $\delta$ is the distance between $\lambda$ and the Ritz value other that $\mu$, i.e.,
\eq{eq:delta0}
\delta = \displaystyle \min_{\mu_j \in \Theta \setminus \mu}|\lambda - \mu_j|.
\en
\end{theorem}
Throughout, 
$\|\cdot\|$ denotes either the spectral or the Frobenius norm of a matrix; or a~vector's 2-norm, 
depending on the context. The matrix Frobenius norm will be denoted by $\|\cdot\|_F$.

%
Bound~\eqref{eq:saad} is often referred to as ``Saad's bound'' in literature, 
e.g.,~\cite{Chen.Jia:04,Stewart:01}. It was later extended by Stewart~\cite{Stewart:01} to invariant
subspaces of general matrices. 
\vspace{-.5cm}
\begin{theorem}[Stewart~\cite{Stewart:01}]\label{thm:stewart}
Let $\calX$ be an invariant subspace of a (possibly non-Hermitian) matrix $A$.
Let $\calU$ be a Ritz subspace\footnote{Let $(M,U)$ be a matrix pair, such that
all columns of $U$ are in $\calK$ and $AU - UM \perp \calK$. Then the Ritz subspace 
$\calU \subseteq \calK$ is defined as a column space of $U$. If $A$ is Hermitian, then $\calU$ is 
a subspace spanned by Ritz vectors.} 
and $\calV$ its orthogonal~complement in $\calK$.  
Then
\eq{eq:stewart}
\sin \angle (\calX, \calU) \leq \sqrt{1 + \frac{\gamma^2}{\delta^2}} \sin \angle(\calX, \calK),
\en
with $\gamma = \|P_{\calK} A (I - P_{\calK} ) \|$
and $\delta$ defined by  
\eq{eq:delta0_sep}
\delta = \displaystyle \inf_{\|Z\| = 1}\|(V^* A V) Z  - Z (X^* A X)\|,
\en
where $X$ and $V$ are arbitrary orthonormal bases of $\calX$ and $\calV$, respectively.
\end{theorem}

The angle between two subspaces in~\eqref{eq:stewart} is defined as
\eq{eq:subsp_angle}
\angle (\calX, \calK) = \min_{x \in \calX, x\neq0 \atop y \in \calK, y\neq0} \angle (x,y).
\en
Note that if $A$ is Hermitian and $\calX$, $\calU$ are
one-dimensional subspaces spanned by an eigenvector $x$ and a Ritz vector $u$, respectively,
then 
the values of $\delta$ in~\eqref{eq:delta0} and~\eqref{eq:delta0_sep} coincide;
and Theorem~\ref{thm:stewart} reduces to Theorem~\ref{thm:saad}. 

The Ritz pairs $(\mu,u)$
are known to be best 
suited for approximating the extreme eigenpairs of $A$, i.e., those $(\lambda,x)$ that correspond to $\lambda$ near the boundary of $A$'s spectrum, further denoted by $\Lambda(A)$.
If interior eigenpairs are wanted, 
then the Rayleigh--Ritz procedure may not be appropriate; 
it can produce ``spurious'' or ``ghost'' Ritz values~\cite{Morgan:91, Scott:82, Stewart-eigen-book}.

This problem, however, can be fixed by the use of the \textit{harmonic} Rayleigh--Ritz 
procedure~\cite{Morgan:91, Morgan.Zeng:98, Stewart-eigen-book}. 
Given a shift $\sigma$ pointing to a location inside $\Lambda(A)$,
the harmonic Rayleigh--Ritz scheme aims at finding the harmonic Ritz pairs 
$(\theta,v)$
that approximate the eigenpairs $(\lambda,x)$ of~$A$ 
associated with the eigenvalues $\lambda$ closest to $\sigma$. This is fulfilled by imposing
the Petrov--Galerkin condition
\eq{eq:galerkin}
A v - \theta v \perp (A - \sigma I) \calK, \quad v \in \calK,
\en
which, similar to~\eqref{eq:rr}, gives an $s$-by-$s$ eigenvalue problem.
In particular, if $A$ is Hermitian, this eigenvalue problem is of the form
\eq{eq:hrr}
K^* (A - \sigma I)^2 K c = \xi K^* (A - \sigma I) K c.
\en
The eigenpairs $(\xi,c)$ of~\eqref{eq:hrr} yield the harmonic Ritz pairs $(\theta, v)$,
where $\theta = \xi + \sigma$ is a harmonic Ritz value and 
$v = Kc$ is the corresponding harmonic Ritz~vector.

In this paper, we present a Saad's type bound for harmonic Ritz vectors~of a Hermitian 
matrix $A$.  
It shows that, along with $\angle(x,\calK)$, 
the closeness of~the harmonic Ritz vectors to the exact eigenvectors
generally depends on 
the spectral condition number of $A - \sigma I$. 
This property of the harmonic Rayleigh--Ritz 
procedure is fundamentally different from the standard
Rayleigh--Ritz which,~according to Theorem~\ref{thm:saad},
is not affected by conditioning of the (shifted)
operator.

Our finding has a practical implication.
Namely, difficulties related to a~poor conditioning of
algebraic systems are commonly mitigated by the use of preconditioners; e.g.,~\cite{Greenbaum:97, Saad:03}. 
Therefore, motivated by the dependence on the condition number, 
one can expect to improve the harmonic Rayleigh--Ritz approximations 
by properly preconditioning the procedure.


A possible way to blend preconditioning directly into the harmonic Rayleigh--Ritz
scheme was proposed in~\cite{Ve.Kn:14TR}.
There, 
the authors 
introduce the \textit{$T$-harmonic} Rayleigh--Ritz procedure, which is
defined by the Petrov--Galerkin condition~\eqref{eq:galerkin} with respect to the inner product 
$(\;\cdot\;,\;\cdot\;)_T = (\;\cdot\;, T \; \cdot)$, where $T$ is a Hermitian positive definite (HPD) 
preconditioner. Within this framework, the eigenpairs
of $A$ are approximated by the $T$-harmonic Ritz pairs $(\theta,v)$, such that  
\eq{eq:galerkinT}
A v - \theta v \perp_T (A - \sigma I) \calK, \quad v \in \calK,
\en
where $\perp_T$ denotes orthogonality in the $T$-inner product. If $A$ is Hermitian, 
the procedure amounts to solving an $s$-by-$s$ eigenvalue problem 
\eq{eq:hrrT}
K^* (A - \sigma I) T (A - \sigma I) K c = \xi K^* (A - \sigma I) T K c.
\en
The $T$-harmonic Ritz values are then given by $\theta = \xi + \sigma$, whereas the $T$-harmonic
Ritz vectors are defined by $v = K c$. 

In the present work, we address an idealized situation where the HPD preconditioner $T$
commutes with $A$. In this case, our generalization of the Saad's bound can further 
be easily extended to the $T$-harmonic Rayleigh--Ritz. 
We show that, along with $\angle(x,\calK)$, 
the proximity of the $T$-harmonic Ritz vectors to the exact eigenvectors 
depends on the condition number of the matrix $T^{1/2}(A - \sigma I)$. 
In particular, this means that the approximation quality can be improved in practice by
properly choosing a preconditioner $T$. 
We briefly discuss several possibilities for defining $T$, including the absolute value 
preconditioning~\cite{thesis, Ve.Kn:13}. 

Finally, we note that other generalizations of the Saad's bound on the harmonic Ritz 
vectors were obtained in~\cite{Chen.Jia:04, Jia:04}.
These results, 
however, aim at~general matrices $A$
and as a consequence do not capture certain peculiarities of the Hermitian case.
In particular, the bounds in~\cite{Chen.Jia:04, Jia:04} do not reveal the dependence
on the condition number. Furthermore, they fail to imply that the 
harmonic Ritz vectors necessarily converge to the exact eigenvectors as $\angle(x,\calK)$ 
decreases.

The paper is organized as following. Section~\ref{sec:harm} presents our main result. 
Related work, such as~\cite{Chen.Jia:04, Jia:04}, is discussed in Section~\ref{sec:relres}.  
Section~\ref{sec:mult} provides an extension of the main theorem on eigenspaces associated with multiple 
eigenvalues.
In Section~\ref{sec:tharm},
we consider the case of the $T$-harmonic Rayleigh--Ritz and derive the Saad's bound 
for commuting $A$ and $T$.
Throughout, we assume that $A$ is Hermitian. 

\section{A bound for harmonic Ritz vectors}\label{sec:harm}

We start with a lemma that provides a two-sided bound on the angle between an 
eigenvector $x$ and an arbitrary vector $y$ in terms of $\angle(x,Ay)$.
This bound will be crucial for deriving our main result. 

\begin{lemma}\label{lem:sin}
Let $(\lambda,x)$ be an eigenpair of a nonsingular Hermitian matrix $A$. 
Then for any vector~$y$, we have
\eq{eq:sin}
| \lambda/\lambda_{\max} | \sin\angle(x,Ay) \leq \sin\angle(x,y) \leq 
| \lambda/\lambda_{\min} | \sin\angle(x,Ay), 
\en
where $\lambda_{\min}$ and $\lambda_{\max}$ are the smallest and largest magnitude eigenvalues
of~$A$, respectively.
\end{lemma}
\begin{proof}
Let us fist introduce the notation $\phi = \angle(x,y)$, $\phi_A = \angle(x,A y)$ 
and observe that $\phi_A = 0$ if and only if 
$\phi = 0$. Hence, if $\phi = 0$ then bound~\eqref{eq:sin} is trivial. 
Therefore, in what follows, we consider only the case where $0 < \phi \leq \pi/2$. 

Without loss of generality, we assume that both $x$ and $y$ are unit vectors.
Then, since $x$ is an eigenvector, we observe that 
\begin{eqnarray*}
\cos\phi = |x^* y|  = \left| \frac{x^* A y}{\lambda} \right| = \left( \frac{\|Ay\|}{|\lambda|} \right) \left( \frac{|x^* A y|}{\|Ay\|} \right)
 = \left( \frac{\|Ay\|}{|\lambda|} \right) \cos \phi_A, 
\end{eqnarray*}
where
$\cos \phi = |x^* y|$ and $\cos \phi_A = |x^* A y| /\|Ay\|$. This relation 
implies that
\eq{eq:sin_ratio}
\frac{\sin^2 \phi}{\sin^2 \phi_A} = \frac{\sin^2 \phi}{1 - \cos^2 \phi_A} =  
\frac{\|Ay\|^2 \sin^2 \phi}{\|Ay\|^2 - \lambda^2 \cos^2 \phi}.
\en

Let $y = x (x^*y) + w (w^* y)$ be a representation of $y$ in terms of the eigenvector~$x$ 
and a unit vector $w$ orthogonal to $x$. Then
$\|Ay\|^2 = \lambda^2 \cos^2\phi +  \left(w^* A^2 w\right) \sin^2\phi$,
where $\sin^2\phi = |w^* y|^2 = \cos^2\angle(w,y)$.  
Substituting this expression into the right-hand side 
of~\eqref{eq:sin_ratio} gives
\eq{eq:sin_ratio1}
\frac{\sin^2 \phi}{\sin^2 \phi_A} = \left( \frac{\lambda^2}{w^* A^2 w} \right) \cos^2 \phi + \sin^2 \phi.
\en

By the Courant-Fischer theorem~\cite{Parlett:98, Saad-book3}, 
$a_0^2 \leq \lambda^2/(w^* A^2 w) \leq a_1^2$,  
where 
\eq{eq:a0}
a_0^2 = \min_{q \in x^{\perp}, \atop \|q\| = 1} \frac{\lambda^2}{q^* A^2 q} = 
\frac{\lambda^2}{\displaystyle \max_{\lambda_j \in \Lambda(A) \setminus \lambda} \lambda_j^2},
\en
and
\eq{eq:a1}
a_1^2 = \max_{q \in x^{\perp}, \atop \|q\| = 1} \frac{\lambda^2}{q^* A^2 q} =  
\frac{\lambda^2}{\displaystyle \min_{\lambda_j \in \Lambda(A) \setminus \lambda} \lambda_j^2}.
\en
Thus, from~\eqref{eq:sin_ratio1}--\eqref{eq:a1}, we obtain
\eq{eq:sin_ratio2}
a_{0}^2 \cos^2 \phi + \sin^2 \phi \leq
\frac{\sin^2 \phi}{\sin^2 \phi_A} \leq 
a_{1}^2 \cos^2 \phi + \sin^2 \phi.
\en

Let us now consider the function $f(z; a) = a^2 \cos^2 z + \sin^2 z$, where 
$z$ is a variable and $a^2$ is a fixed positive parameter.
Then~\eqref{eq:sin_ratio2}~can be written as
\eq{eq:sin_ratio3}
f(\phi; a_{0}^2) \leq 
\frac{\sin^2 \phi}{\sin^2 \phi_A} \leq
f( \phi; a_{1}^2).
\en
Hence, for any
$0 < \phi \leq  \pi/2$, 
\eq{eq:sin_ratio4}
\min_{z \in [0,\pi/2]} f(z; a_{0}^2) \leq 
\frac{\sin^2 \phi}{\sin^2 \phi_A} \leq 
\max_{z \in [0,\pi/2]} f(z; a_{1}^2).
\en

It is easy to check, 
by differentiation, 
that $f(z;a^2)$ is monotonically increasing
on $[0,\pi/2]$ if $a^2 \leq 1$. If $a^2 \geq 1$, then the function is  
decreasing. 

From~\eqref{eq:a0}, we see that $a_0^2 < 1$ if $\lambda \neq \lambda_{\max}$, 
where $\lambda_{\max}$ is an eigenvalue of~$A$ of the largest absolute
value. Therefore, in this case, $f(z;a_0^2)$ is increasing on $[0,\pi/2]$
and its minimum 
is given by 
$f(0; a_{0}^2) = \lambda^2/\lambda^2_{\max}$.
At the same time, if  $\lambda = \lambda_{\max}$, then $a_0^2 \geq 1$,
and, hence, $f(z; a_{0}^2)$ is decreasing on $[0,\pi/2]$. 
Therefore, the minimum is delivered by $f(\pi/2; a_{0}^2) = 1$. Thus, we get
\eq{eq:fmin}
\min_{z \in [0,\pi/2]} f(z; a_{0}^2) = \left\{
\begin{array}{cl}
\lambda^2/\lambda_{\max}^2, & \mbox{if} \ \lambda \neq \lambda_{\max}, \\
1, & \mbox{if} \ \lambda = \lambda_{\max}. \\
\end{array}
\right.
\en 
After combining the both cases in~\eqref{eq:fmin}, we conclude that
\eq{eq:lhs_bound}
\min_{z \in [0,\pi/2]} f(z; a_{0}^2) = \lambda^2/\lambda_{\max}^2. 
\en 
 

Similarly, by~\eqref{eq:a1}, $a_1^2 > 1$ if $\lambda \neq \lambda_{\min}$, 
where $\lambda_{\min}$ denotes an eigenvalue of~$A$ of the smallest absolute
value; and $a_1^2 \leq 1$ otherwise. By applying exactly the same 
argument, based on the monotonicity of $f(z;a_1^2)$, as above, we obtain
\eq{eq:rhs_bound}
\max_{z \in [0,\pi/2]} f(z; a_{1}^2) = \lambda^2/\lambda_{\min}^2.
\en 
Substituting~\eqref{eq:lhs_bound} and~\eqref{eq:rhs_bound} 
into~\eqref{eq:sin_ratio4}
and taking the square root of all parts of the inequality gives~\eqref{eq:sin}.
\end{proof}

Note that Lemma~\ref{lem:sin} suggests that, in particular, if $(\lambda,x)$ is an 
eigenpair corresponding to the smallest magnitude eigenvalue, then $\angle(x,y) \leq \angle(x,Ay)$,
i.e., the approximation quality of a vector $y$ does not improve after multiplication with $A$.
On the other hand, if $(\lambda,x)$ is associated with the largest magnitude eigenvalue, then
$\angle(x,y) \geq \angle(x,Ay)$. The latter is not surprising, because a step of the power method
applied to $y$ is expected to yield a better approximation to the dominant eigenvector. 

Given a subspace $\calK$, the following corollary relates $\angle(x,\calK)$ 
and~$\angle(x,A\calK)$. 

\begin{corollary}\label{cor:sin_subsp}
Let $(\lambda,x)$ be an eigenpair of a nonsingular Hermitian matrix $A$. Then for
any subspace $\calK$ of $\IC^n$, we have
\begin{equation}\label{eq:sin_subsp}  
\sin \angle (x,A\calK) \leq |\lambda_{\max}/\lambda| \sin \angle (x,\calK). 
\end{equation}  
\end{corollary}
\begin{proof}
From the left-hand side of~\eqref{eq:sin}, we have 
\[
\sin\angle(x,Ay) \leq |\lambda_{\max}/\lambda| \sin\angle(x,y).
\]
This inequality holds for any vector $y$. In particular, it is true for some $y_* \in \calK$
that yields the minimum of $\angle(x,y)$ over all $y$ in~$\calK$.
By definition~\eqref{eq:sin_def}, $\angle(x,y_*)$
is exactly the angle between the vector $x$ and 
the subspace~$\calK$. 
Thus, 
we obtain
\eq{eq:bnd3}
\sin \angle (x, A y_*) \leq |\lambda_{\max}/\lambda| \sin \angle (x, \calK). 
\en 
On the other hand,
\[
\angle (x, A \calK) = \min_{y \in \calK, y\neq 0}\angle(x,Ay) \leq \angle (x, A y_*).
\]
Therefore, $\sin \angle (x, A \calK) \leq \sin \angle (x, A y_*)$. Combining this inequality
with~\eqref{eq:bnd3} leads to~\eqref{eq:sin_subsp}, which completes the proof. 
\end{proof}

We are now ready to state the main result.
\begin{theorem}\label{thm:hsaad}
Let $(\lambda, x)$ be an eigenpair of a Hermitian matrix $A$ and
$(\theta, v)$ be a harmonic Ritz pair with respect to the subspace $\calK$
and shift $\sigma \notin \Lambda(A)$.
Assume that~$\Theta$ is a set of all the harmonic Ritz values and 
let $P_{\calQ}$ be an orthogonal projector onto $\calQ = (A - \sigma I) \calK$
. 
Then
\eq{eq:hsaad}
\sin \angle (x, v) \leq \kappa(A - \sigma I) \sqrt{1 + \frac{\gamma^2}{\delta^2}} \sin \angle(x, \cal K),
\en
where
$\gamma = \|P_{\calQ} (A - \sigma I)\inv (I - P_{\calQ} ) \|$, 
\eq{eq:kappa}
\kappa(A - \sigma I) = \frac{\displaystyle \max_{\lambda_j \in \Lambda(A)}|\lambda_j - \sigma|}
{\displaystyle \min_{\lambda_j \in \Lambda(A)}|\lambda_j - \sigma|},
\en
and
\eq{eq:gamma_delta}
\delta = \displaystyle \min_{\theta_j \in \Theta \setminus \theta} \left| \frac{\theta_j - \lambda}
{(\lambda - \sigma)(\theta_j - \sigma)} \right|.
\en
\end{theorem}
\begin{proof}
We first observe that the eigenvalue problem~\eqref{eq:hrr} can be 
formulated~as
\eq{eq:hrr2}
(SK)^* S\inv (SK) c = \tau (SK)^*(SK)c, \quad S = A - \sigma I,
\en   
where the matrix $S$ is nonsingular because $\sigma \notin \Lambda(A)$. 
Each  
eigenpair $(\tau, c)$ of~\eqref{eq:hrr2} 
yields an eigenpair $(\xi,c)$ 
of~\eqref{eq:hrr} with $\xi = 1/\tau$. Thus, given $(\tau, c)$,
the corresponding harmonic Ritz pair $(\theta, v)$ is defined by $\theta = 1/\tau + \sigma$ and $v = Kc$. 

At the same time, problem~\eqref{eq:hrr2} 
corresponds to
the Rayleigh--Ritz procedure for the matrix $S\inv$ with respect to the
subspace $\calQ = S\calK$, in which case a Ritz pair is given by $(\tau, Sv)$,
where $v = K c$ is the harmonic Ritz vector. 
The eigenvalues of $S\inv$ are related to the eigenvalues $\lambda$ of $A$ as
$1/(\lambda-\sigma)$, and the corresponding eigenvectors $x$ coincide. 
Then Theorem~\ref{thm:saad} guarantees that for an eigenpair 
$(1/(\lambda-\sigma), x)$ of $S^{-1}$ and a Ritz pair $(\tau, Sv)$, 
\eq{eq:hsaad0}
\sin \angle (x, Sv) \leq \sqrt{1 + \frac{\gamma^2}{\delta^2}} \sin \angle(x, \calQ),
\en   
with $\gamma = \|P_{\calQ} S\inv (I - P_{\calQ} ) \|$ and,
since $\tau = 1/(\theta - \sigma)$,
\[
\delta = \min_{\theta_j \in \Theta \setminus \theta} 
|\frac{1}{\lambda-\sigma} - \frac{1}{\theta_j - \sigma}| = 
\displaystyle \min_{\theta_j \in \Theta \setminus \theta} \left| \frac{\theta_j - \lambda}
{(\lambda - \sigma)(\theta_j - \sigma)} \right|.
\]  

Clearly, if $(\lambda,x)$ is an eigenpair of $A$ then $(\lambda-\sigma, x)$
is an eigenpair of $S$. Therefore, recalling
that $\calQ = S\calK$, we can apply Corollary~\ref{cor:sin_subsp} with respect to
$S$ to bound $\sin\angle(x,\calQ)$ in~\eqref{eq:hsaad0} from above by a term proportional to 
$\sin\angle(x,\calK)$. 
As a result, from~\eqref{eq:hsaad0}, we get
\eq{eq:hsaad1}
\sin \angle (x, Sv) \leq \frac{\displaystyle \max_{\lambda_j \in \Lambda(A)}|\lambda_j - \sigma|}
{|\lambda - \sigma|} \sqrt{1 + \frac{\gamma^2}{\delta^2}} \sin \angle(x, \calK).
\en   
On the other hand, by Lemma~\ref{lem:sin}, also applied with respect to $S$, we obtain 
\eq{eq:hsaad2} 
\sin\angle(x,v) \leq \frac{|\lambda - \sigma|}{\displaystyle \min_{\lambda_j \in \Lambda(A)}|\lambda_j - \sigma|} \sin\angle(x,Sv). 
\en 
The desired bound~\eqref{eq:hsaad} 
then follows from~\eqref{eq:hsaad1} and~\eqref{eq:hsaad2}. 
\end{proof}

Theorem~\ref{thm:hsaad} shows that the approximation quality of 
the harmonic Rayleigh--Ritz procedure can be hindered by a poor conditioning
of $A - \sigma I$. In particular, this can happen if the shift $\sigma$
is chosen to be close to an eigenvalue of $A$.  

Furthermore, the structure of the quantity $\delta$ in~\eqref{eq:gamma_delta} 
suggests that the proximity of the harmonic Ritz vectors to exact eigenvectors 
can be affected by~clustering of $A$'s eigenvalues, 
in which case $\delta$ can be close to zero. 
The smallness of $\delta$ can also be caused by a 
large difference $|\lambda - \sigma|$ in the denominator of~\eqref{eq:gamma_delta}.~This indicates that   
the harmonic Rayleigh--Ritz scheme should be most efficient for 
approximating the eigenpairs associated with the eigenvalues 
closest to a given shift $\sigma$. 
Note that 
$\delta$ is finite, as 
$\sigma \neq \lambda$ and $\theta_j \neq \sigma$~\cite[Theorem 2.2]{Morgan.Zeng:98}. 

Bound~\eqref{eq:hsaad} implies that a harmonic Ritz vector $v$ must approach 
the exact eigenvector $x$ as the angle between $x$ and the subspace $\calK$ decreases,
provided that there is only one harmonic Ritz value that converges to the targeted 
eigenvalue $\lambda$. 
In the opposite case, which can occur if $\lambda$ is a multiple eigenvalue,
the quantity~$\delta$ converges to zero 
(the set $\Theta$ in~\eqref{eq:gamma_delta} assumes repetition of multiple harmonic
Ritz values). Hence, in this setting, 
bound~\eqref{eq:hsaad} may not guarantee that $\angle(x,v)$ is small whenever 
$\angle(x,\calK)$ is sufficiently small. 
%
This limitation, however, is natural as it reflects the fact that the direction of $x$ is 
not unique in the case of a multiple $\lambda$ and that the harmonic Ritz vector 
can tend to approximate any other element of the associated eigenspace.
A proper extension of Theorem~\ref{thm:hsaad}, which gives a meaningful bound in the case
where $\lambda$ has multiplicity greater than 1,  
will be considered below in Section~\ref{sec:mult}. 
%




The result of Theorem~\ref{thm:hsaad} is very general in that it 
holds for any choice of the subspace $\calK$. Hence, it is rather pessimistic.
In particular, practical eigensolvers
construct the subspace $\calK$,
often called the trial or search subspace, very carefully, 
in such a way that it does not contain 
contributions from unwanted eigenvectors.   

We address this practical setting in the following corollary. 
It shows that if $\calK$ is chosen from an invariant subspace of $A$
associated with only a part of its spectrum, which however contains the wanted eigenvalues,
then the condition number in the right-hand side of~\eqref{eq:hsaad} can be reduced.

\begin{corollary}\label{cor:hsaad_defl}
Let $X$ be a matrix whose  columns represent an orthonormal basis of an invariant subspace  
of $A$ associated with a subset 
$\Lambda_X(A) \subseteq \Lambda(A)$ of its eigenvalues, and assume that $\calK \subseteq \mbox{\emph{range}}(X)$.
Let $(\lambda, x)$ be an eigenpair of $A$, such that~$\lambda \in \Lambda_X(A)$, and
let $(\theta, v)$ be a harmonic Ritz pair with respect to $\calK$
and $\sigma \notin \Lambda_X(A)$.
Assume that~$\Theta$ is a set of all the harmonic Ritz values and 
let $P_{\calQ}$ be an orthogonal projector onto $\calQ = X^* (A - \sigma I) \calK$
. 
Then
\eq{eq:hsaad_defl}
\sin \angle (x, v) \leq \kappa(X^* A X - \sigma I) \sqrt{1 + \frac{\gamma^2}{\delta^2}} \sin \angle(x, \cal K),
\en
where
$\gamma = \|P_{\calQ} (X^*AX - \sigma I)\inv (I - P_{\calQ} ) \|$, 
\eq{eq:kappa_defl}
\kappa(X^* A X - \sigma I) = \frac{\displaystyle \max_{\lambda_j \in \Lambda_X(A)}|\lambda_j - \sigma|}
{\displaystyle \min_{\lambda_j \in \Lambda_X(A)}|\lambda_j - \sigma|},
\en
and $\delta$ is defined in~\eqref{eq:gamma_delta}.
\end{corollary}
\begin{proof}
Since $\calK \subseteq \mbox{range} (X)$, a basis $K$ of this subspace
can be expressed as $K = X W$, where $W$ is a $k$-by-$s$  
matrix, with $k$ being the number of columns in $X$ and $s = \dim(\calK)$.
Substituting $K = X W$ into~\eqref{eq:hrr}, and using the fact that 
the columns of $X$ are orthonormal and span an invariant subspace of $A$, 
gives 
\eq{eq:rrX}
W^* (X^* A X - \sigma I)^2 W c = \xi W^*(X^* A X - \sigma I) W c. 
\en   
This eigenvalue problem corresponds to the 
harmonic Rayleigh--Ritz
procedure for the matrix $X^*AX$ with respect to the subspace 
$\calW = \mbox{range}(W) \subseteq \IC^k$
and shift $\sigma$. The eigenvalues of $X^* A X$ are exactly the eigenvalues $\lambda$
of $A$ in $\Lambda_X(A)$, whereas the associated eigenvectors $\hat x$
are related to those of $A$ by $x = X \hat x$. 

The harmonic Ritz pairs $(\theta, \hat v)$ of $X^*AX$ 
are defined by the eigenpairs of~\eqref{eq:rrX}, such that $\theta = \xi + \sigma$ and 
$\hat v = W c$. Then, by Theorem~\ref{thm:hsaad}, applied with respect to
$X^*AX$ and $\calW$,  
for each eigenpair $(\lambda, \hat x)$ of $X^*AX$ and a harmonic Ritz pair~$(\theta, \hat v)$,
\eq{eq:hsaad_defl0}
\sin\angle(\hat x, \hat v)\leq\kappa(X^*AX - \sigma I) \sqrt{1 + \frac{\gamma^2}{\delta^2}}\sin\angle(\hat x, \calW),
\en  
where $\kappa(X^*AX - \sigma I)$ is defined in~\eqref{eq:kappa_defl} and 
$\gamma = \|P_{\calQ} (X^*AX - \sigma I)\inv (I - P_{\calQ} ) \|$ with 
$\calQ = (X^* A  X - \sigma I)\calW = X^* (A - \sigma I) X\calW = X^* (A - \sigma I) \calK$,
since $\calK = X\calW$.
The quantity $\delta$ is given by~\eqref{eq:gamma_delta}, where the set 
$\Theta$ of the harmonic Ritz values of $X^*AX$ with respect to $\calW$
coincides with the harmonic Ritz values of $A$ over~$\calK$.   
But $\angle(\hat x, \hat v) = \angle(X \hat x, X \hat v) = \angle(x, v)$, since
$X$ has orthonormal columns and $X \hat v = X W c = Kc = v$, where $v$ is a harmonic Ritz
vector of~$A$ with respect to $\calK$ associated with the harmonic Ritz value $\theta$. 
Similarly, $\angle(\hat x, \calW) = \angle(X \hat x, X \calW) = \angle(x, \calK)$. 
Thus,~\eqref{eq:hsaad_defl} follows from~\eqref{eq:hsaad_defl0}.
\end{proof}

In particular, Corollary~\ref{cor:hsaad_defl} implies that choosing 
$\calK$ from the invariant subspace of $A$ associated with the eigenvalues 
$\{ \lambda_1, \lambda_2, \ldots, \lambda_k \}$
that are closest to~$\sigma$, such that 
$|\lambda_1 - \sigma| \leq | \lambda_2 - \sigma | \leq \ldots \leq |\lambda_k - \sigma |$, 
yields the effective condition number of $|\lambda_k - \sigma|/|\lambda_1 - \sigma|$, which can be 
much lower than $\kappa(A - \sigma I) = |\lambda_n - \sigma|/|\lambda_1 - \sigma|$ 
suggested by Theorem~\ref{thm:hsaad},
where $\lambda_n$ is an eigenvalue of $A$ that is the most distant from $\sigma$. 
In practice, such a choice of $\calK$ is achieved by damping out the unwanted 
eigenvector components from a trial subspace, e.g., using filtering or preconditioning 
techniques; 
e.g.,~\cite{Fang.Saad:12, Knyazev:01, Sleijpen.Vorst:96, Tang.Polizzi:14, Ve.Kn:14TR}. 


\section{Related work}\label{sec:relres} 

Other bounds for the harmonic Ritz vectors were established in~\cite{Chen.Jia:04, Jia:04}.  
These results are more general than~\eqref{eq:hsaad} 
in that they hold for any $A$, which can be non-Hermitian. 
However,
as we will see below, 
in the Hermitian case,~which is the focus of this paper,
the presented bound~\eqref{eq:hsaad} turns out to be more descriptive.  

In particular, if $A$ is Hermitian, the result of~\cite{Chen.Jia:04} states
that
\eq{eq:chen.jia}
\sin \angle (x, v) \leq  \sqrt{1 + \frac{\gamma_1^2 \|B^{-1}\|^2}{\mbox{sep}(\lambda,G)^2}} \sin \angle(x, \cal K),
\en
where $B = K^* (A - \sigma I) K$, $\gamma_1 = \|P_{\calK} (A - \sigma I) (A - \lambda I) (I - P_{\calK} ) \|$, and $P_{\calK}$ denotes an orthogonal projector onto $\calK$. The quantity
$$
\mbox{sep}(\lambda,G) = \|(G - \lambda I)^{-1}\|^{-1}
$$
describes separation of the targeted eigenvalue $\lambda$
from the harmonic Ritz values different from the one associated with $v$,
which are the eigenvalues of $G$; 
see~\cite{Chen.Jia:04} for the precise definition of $G$.

The result of~\cite{Jia:04} suggests that 
\eq{eq:jia}
\sin \angle (x, v) \leq \left( 1 + \frac{2 \|B^{-1}\|^2 \gamma_2^2}{\sqrt{1 - \sin^2 \angle(x,\calK)} \mbox{sep}(\lambda,G)} \right) \sin \angle(x, \cal K),
\en
where, if $A$ is Hermitian, the matrix $B = K^* (A - \sigma I) K$ is the same as in~\eqref{eq:chen.jia} and 
$\gamma_2$ is the maximum value of $|\lambda - \sigma|$ over all eigenvalues $\lambda$ of $A$. 
Similarly, $\mbox{sep}(\lambda,G)$ gives a separation of 
$\lambda$ from the unwanted harmonic Ritz values.

Both~\eqref{eq:chen.jia} and~\eqref{eq:jia} share a number of similarities with 
the bound~\eqref{eq:hsaad} of this paper. In particular, all of them show that the 
approximation quality of a Ritz vector depends on $\angle(x,\calK)$,
separation of $\lambda$, 
and the choice of the shift $\sigma$, which should 
not be too close to an eigenvalue. 

However, the main difference of~\eqref{eq:hsaad} 
is that it eliminates the dependence of the bound on the norm of $B^{-1} = (K^* (A - \sigma I) K)^{-1}$. 
This norm can generally be large or unbounded, even if $\sigma$ is well chosen, 
the subspace $\calK$ contains a good eigenvector approximation, 
and the corresponding eigenvalue is well separated.
%
As a result, bound~\eqref{eq:hsaad} guarantees 
that, if $\lambda$ is a simple eigenvalue, a harmonic Ritz vector $v$ must converge to the eigenvector $x$
as the angle between~$\calK$ and $x$ decreases 
(the same conclusion for eigenvalues of a higher multiplicity is obtained in the next section). 
By contrast, neither~\eqref{eq:chen.jia} nor~\eqref{eq:jia} can lead to this conclusion 
without an additional assumption on the uniform boundedness
of $\|B\|^{-1}$; see~\cite{Chen.Jia:04,Jia:04}.

\section{Extension on eigenspaces}\label{sec:mult}

As has already been discussed in Section~\ref{sec:harm}, bound~\eqref{eq:hsaad}
is not useful if the targeted eigenvector $x$ corresponds 
to an eigenvalue $\lambda$ of multiplicity greater~than~1. In this case,
instead of an individual eigenvector $x$,
the focus should be shifted on the eigenspace $\calX$ associated with $\lambda$. 
In particular, a question to ask is whether there exist a subspace $\calV$ 
spanned by harmonic Ritz vectors 
extracted from $\calK$,
which gives a good approximation to $\calX$, provided that 
$\calK$ contains a good approximation to the wanted eigenspace $\calX$.
 
We note that a similar limitation is also true for the original Saad's
bound of Theorem~\ref{thm:saad}. Fortunately, the Stewart's Theorem~\ref{thm:stewart}
suggests an appropriate extension on eigenspaces, which is stated in the following corollary.    

\begin{corollary}\label{cor:stewart}
Let $\calX$ be an eigenspace associated with an eigenvalue $\lambda$ of $A$, and
let 
$\calU$ be a subspace spanned by Ritz vectors associated with Ritz
values $\Theta_k = \left\{ \mu_1, \mu_2, \ldots, \mu_k \right\}$ 
with respect to~$\calK$. 
Assume that~$\Theta$ is a set of all the Ritz values and 
let $P_{\calK}$ be an orthogonal projector onto $\calK$.
Then
\eq{eq:stewart_mult}
\sin \angle (\calX, \calU) \leq \sqrt{1 + \frac{\gamma^2}{\delta^2}} \sin \angle(\calX, \calK),
\en
where $\gamma = \|P_{\calK} A (I - P_{\calK} ) \|_F$
and $\delta$ is defined by  
\eq{eq:delta0_sep_mult}
\delta = \displaystyle \min_{\mu_j \in \Theta \setminus \Theta_k }|\lambda - \mu_j|.
\en
\end{corollary}
\begin{proof}
We apply Theorem~\ref{thm:stewart} to the subspaces $\calX$ and $\calU$, such that 
$\calX$ is an eigenspace of $\lambda$ and $\calU$ is the Ritz subspace 
associated with $\Theta_k$; and 
choose~$\|\cdot\|$ to be~the~Frobenius norm. This immediately yields bound~\eqref{eq:stewart_mult},
where $\gamma = \|P_{\calK} A (I - P_{\calK} ) \|_F$. It then only remains to determine the value
of $\delta$.

Using the fact that $\calX$ is an eigenspace of $\lambda$, from~\eqref{eq:delta0_sep},  we obtain  
\[
\delta^2 = \displaystyle \min_{\|Z\|_F = 1}\|(V^* A V) Z  - Z (X^* A X)\|_F^2 =
\displaystyle \min_{\|Z\|_F = 1}\|(V^* A V) Z  - \lambda Z\|_F^2, 
\]
where the columns of $X$ and $V$ represent orthonormal bases of~$\calX$ (i.e., $AX = \lambda X$) 
and of the orthogonal complement of $\calU$ in $\calK$, respectively. 
If $m$ is the multiplicity of $\lambda$
and $s$ is the dimension of $\calK$, then $Z$ is an $(s-k)$--by--$m$ matrix.   
Thus, 
the above equality can be written as  
\eq{eq:eq0}
\begin{array}{ccc}  
\delta^2 
& = & \displaystyle \min_{\|z_1\|^2 + \ldots + \|z_m\|^2 = 1} \sum_{i=1}^m \|(V^* A V  - \lambda I) z_i\|^2  \\ 
         & = & 
\displaystyle \min_{\|z_1\|^2 + \ldots + \|z_m\|^2 = 1} \sum_{i=1}^m ( (V^* A V  - \lambda I)^2 z_i, z_i ),   
\end{array}
\en  
where $z_i$ denote the columns of $Z$.
By the Courant-Fischer theorem~\cite{Parlett:98, Saad-book3},  
$$
( (V^* A V  - \lambda I)^2 z_i, z_i ) \geq \zeta^2 \|z_i\|^2, \qquad i = 1,\ldots,m;
$$
where $\zeta^2$ is the smallest eigenvalue of $(V^* A V  - \lambda I)^2$.
Therefore,
\eq{eq:eq1}
\sum_{i=1}^m ( (V^* A V  - \lambda I)^2 z_i, z_i ) \geq \zeta^2 
\sum_{i=1}^m \|z_i\|^2.
\en 
Taking minimum of both sides of~\eqref{eq:eq1} over vectors $z_i$, such that $\sum_{i=1}^m \|z_i\|^2 = 1$, 
gives the inequality 
\[
\min_{\|z_1\|^2 + \ldots + \|z_m\|^2 = 1} \sum_{i=1}^m ( (V^* A V  - \lambda I)^2 z_i, z_i ) 
\geq \zeta^2, 
\] 
which turns into equality if all $z_i$ are set to an eigenvector associated with the
eigenvalue $\zeta^2$, normalized to have a norm of $1/\sqrt{m}$. Hence, 
\[
\min_{\|z_1\|^2 + \ldots + \|z_m\|^2 = 1} \sum_{i=1}^m ( (V^* A V  - \lambda I)^2 z_i, z_i ) 
= \zeta^2, 
\] 
and, from~\eqref{eq:eq0}, we conclude that $\delta^2 = \zeta^2$. 
%
But the eigenvalues of $V^* A V$ are~the Ritz values with respect to $\calK$ that are different from those
in $\Theta_k$. Therefore, $\zeta^2 = \min_{\mu_j \in \Theta \setminus \Theta_k}(\lambda - \mu_j)^2$,
which gives~\eqref{eq:delta0_sep_mult}, and completes the proof.  
\end{proof}

Note that a similar statement can be obtained from Theorem~\ref{thm:stewart} using 
the spectral norm in the definition of $\gamma$ and $\delta$. In this case, one arrives
at bound~\eqref{eq:stewart_mult} with $\gamma = \|P_{\calK} A (I - P_{\calK} ) \|$, where 
$\| \cdot \|$ is the spectral norm. However, the separation constant~$\delta$
will no longer be of the form~\eqref{eq:delta0_sep_mult}, and instead should be determined
according to~\eqref{eq:delta0_sep}, with $X^* A X = \lambda I$, which is somewhat less 
intuitive. For this reason, we prefer to use the Frobenius norm in Corollary~\ref{cor:stewart}.  

In order to extend Theorem~\ref{thm:hsaad} on eigenspaces, we will need the following~result,
which is an immediate corollary of Lemma~\ref{lem:sin}.
\begin{corollary}\label{cor:sin_subsp}
Let $\calX$ be an eigenspace associated with an eigenvalue $\lambda$ of a nonsingular Hermitian matrix $A$. 
Then for any subspace~$\calY$, we have
\eq{eq:sin_mult}
| \lambda/\lambda_{\max} | \sin\angle(\calX,A\calY) \leq \sin\angle(\calX,\calY) \leq 
| \lambda/\lambda_{\min} | \sin\angle(\calX,A\calY), 
\en
where $\lambda_{\min}$ and $\lambda_{\max}$ are the smallest and largest magnitude eigenvalues
of~$A$, respectively. 
\end{corollary}
\begin{proof}
Let vectors $x_* \in \calX$ and $y_* \in \calY$ deliver the minimum of $\angle(x,y)$ 
for all $x \in \calX$ and $y \in \calY$, so that, by definition~\eqref{eq:subsp_angle}, 
$\angle(x_*,y_*) = \angle(\calX,\calY)$.
%
Since~$x_*$ is an eigenvector corresponding to~$\lambda$, we can readily 
apply inequality~\eqref{eq:sin} of Lemma~\ref{lem:sin} with $x = x_*$ and $y = y_*$. 
In particular, the left-hand side of~\eqref{eq:sin}~yields the bound
\eq{eq:sin_left}
| \lambda/\lambda_{\max} | \sin\angle(x_*,A y_*) \leq \sin\angle(\calX,\calY).  
\en  
At the same time, by definition~\eqref{eq:subsp_angle}, 
\[
\angle(\calX,A\calY) = \min_{x \in \calX, x\neq0 \atop y \in \calY, y\neq0} \angle(x,A y) \leq \angle(x_*,A y_*).
\]
Therefore, after combining the above inequality with~\eqref{eq:sin_left}, we obtain
\[
| \lambda/\lambda_{\max} | \sin\angle(\calX,A\calY) \leq | \lambda/\lambda_{\max} | \sin\angle(x_*,A y_*) \leq \sin\angle(\calX,\calY), 
\]  
which proves the left part of~\eqref{eq:sin_mult}. The right part, 
\[
\sin\angle(\calX,\calY) \leq 
| \lambda/\lambda_{\min} | \sin\angle(\calX,A\calY)
\]
is proved analogously by choosing $x_*$ and $y_*$ that
give the minimum of $\angle(x,Ay)$ over $x \in \calX$ and $y \in \calY$ and applying 
the right-hand side of inequality~\eqref{eq:sin} with $x = x_*$ and $y = y_*$.  
\end{proof}

We now state the main result of this section.

\begin{theorem}\label{thm:hsaad_mult}
Let $\calX$ be an eigenspace associated with an eigenvalue $\lambda$ and
let~$\calV$ be a subspace spanned by harmonic Ritz vectors 
associated with harmonic Ritz values $\Theta_k = \left\{ \theta_1, \theta_2, \ldots, \theta_k \right\}$
with respect to the subspace $\calK$ and shift $\sigma \notin \Lambda(A)$.
Assume that~$\Theta$ is a set of all the harmonic Ritz values and 
let $P_{\calQ}$ be an orthogonal projector onto $\calQ = (A - \sigma I) \calK$. 
Then
\eq{eq:hsaad_mult}
\sin \angle (\calX, \calV) \leq \kappa(A - \sigma I) \sqrt{1 + \frac{\gamma^2}{\delta^2}} \sin \angle(\calX, \calK),
\en
where
$\gamma = \|P_{\calQ} (A - \sigma I)\inv (I - P_{\calQ} ) \|_F$, $\kappa(A - \sigma I)$ is 
the condition number defined in~\eqref{eq:kappa}, and
\eq{eq:gamma_delta_mult}
\delta = \displaystyle \min_{\theta_j \in \Theta \setminus \Theta_k} \left| \frac{\theta_j - \lambda}
{(\lambda - \sigma)(\theta_j - \sigma)} \right|.
\en
\end{theorem}
\begin{proof}
As has been established in the proof of Theorem~\ref{thm:hsaad}, the harmonic~Ritz pairs
$(\theta,v)$ of $A$ with respect to $\calK$ and $\sigma$ are related to the Ritz pairs
$(\tau, u)$ of $(A - \sigma)^{-1}$ over the subspace $\calQ$, so that
$\theta = 1/\tau + \sigma$ and $u = (A-\sigma I)v$. Therefore, if 
$\calV$ is a subspace spanned by harmonic Ritz vectors associated with harmonic Ritz values
in $\Theta_k$, then $\calU = (A-\sigma I) \calV$ is a Ritz subspace associated with Ritz values
$\left\{1/(\theta_1 - \sigma), 1/(\theta_2 - \sigma), \ldots, 1/(\theta_k - \sigma)\right\}$ of 
$(A - \sigma)^{-1}$.    
Then, from Corollary~\ref{cor:stewart}, applied to matrix $(A - \sigma)^{-1}$ and 
subspace $\calQ$, we obtain 
\eq{eq:hsaad0_mult}
\sin \angle (\calX, (A-\sigma I)\calV) \leq \sqrt{1 + \frac{\gamma^2}{\delta^2}} \sin \angle(\calX, \calQ),
\en   
where $\gamma = \|P_{\calQ} (A - \sigma I)\inv (I - P_{\calQ} ) \|_F$ and $\delta$ is defined in 
\eqref{eq:gamma_delta_mult}. Applying inequalities~\eqref{eq:sin_mult} of Corollary~\ref{cor:sin_subsp} 
(with $A$ replaced by $A-\sigma I$) to both sides~of~\eqref{eq:hsaad0_mult} leads to the desired
bound~\eqref{eq:hsaad_mult}, where $\kappa(A - \sigma I)$ is defined in~\eqref{eq:kappa}.
\end{proof}

Theorem~\ref{thm:hsaad_mult} shows that if $\lambda$ is a multiple eigenvalue, then there
exists a~subspace $\calV$ spanned by harmonic Ritz vectors that approximates the entire
eigenspace $\calX$ associated with $\lambda$. Moreover, the approximation is improved as 
the angle between $\calK$ and $\calX$ decreases.

\section{A bound for $T$-harmonic Ritz vectors}\label{sec:tharm}

As demonstrated in~\cite{Ve.Kn:14TR}, the robustness of an 
interior eigensolver can be notably improved by incorporating 
a properly chosen HPD preconditioner $T$ into the harmonic Rayleigh--Ritz.
This was done  
by replacing the Petrov--Galerkin condition~\eqref{eq:galerkin} by~\eqref{eq:galerkinT}, 
which lead to the $T$-harmonic Rayleigh--Ritz~procedure.  

The next theorem shows that our main result can be easily~extended to the $T$-harmonic case under the idealized assumption that 
$A$ and $T$ commute. 


\begin{theorem}\label{thm:hsaadT}
Let $(\lambda, x)$ be an eigenpair of a Hermitian matrix $A$ and
$(\theta, v)$ be a $T$-harmonic Ritz pair with respect to the subspace $\calK$ and
shift $\sigma \notin \Lambda(A)$.~Let~$T$ be an HPD preconditioner, such that $TA = AT$.
Assume that~$\Theta$ is a set of all the $T$-harmonic Ritz values and 
let $P_{\calQ}$ be an orthogonal projector onto $\calQ = T^{1/2}(A - \sigma I) \calK$. 
Then
\eq{eq:hsaadT}
\sin \angle (x, v) \leq \kappa(T^{1/2} (A - \sigma I)) \sqrt{1 + \frac{\gamma^2}{\delta^2}} \sin \angle(x, \cal K),
\en
with $\gamma = \|P_{\calQ} (A - \sigma I)\inv (I - P_{\calQ} ) \|$,
$\delta$ defined in~\eqref{eq:gamma_delta}, and
$\kappa(T^{1/2} (A - \sigma I)) = |\nu_{\max}/\nu_{\min}|$, where 
$\nu_{\min}$ and $\nu_{\max}$ are the smallest and largest magnitude eigenvalues of 
$T^{1/2} (A - \sigma I)$, respectively.
\end{theorem}
\begin{proof}
If $A$ and $T$ commute, then~\eqref{eq:hrrT} 
can be written as 
\eq{eq:hrr2T2}
(T^{1/2} SK)^* S\inv (T^{1/2} SK) c = \tau (T^{1/2} SK)^*(T^{1/2} SK)c, \quad S = A - \sigma I.
\en   
This corresponds to the Rayleigh--Ritz procedure for $S\inv$ with respect to the subspace
$\calQ = T^{1/2} S \calK$, where $(\tau, T^{1/2} Sv)$ is a Ritz pair,
and $v = K c$ is the $T$-harmonic Ritz vector.
Thus, Theorem~\ref{thm:saad} applies.
It suggests that for an eigenpair 
$(1/(\lambda-\sigma), x)$ of $S^{-1}$ and a Ritz pair $(\tau, T^{1/2} Sv)$, we have  
\eq{eq:hsaadT0}
\sin \angle (x, T^{1/2}Sv) \leq \sqrt{1 + \frac{\gamma^2}{\delta^2}} \sin \angle(x, \calQ),
\en   
with $\gamma = \|P_{\calQ} S\inv (I - P_{\calQ} ) \|$ and $\delta$ defined in~\eqref{eq:gamma_delta},
where $\Theta$ is the set of all~$T$-harmonic Ritz values with respect to $\calK$,
as $\tau$ is related to $\theta$ by $\tau = 1/(\theta - \sigma)$. 

Since $T$ and $A$ commute, the matrix $T^{1/2}S$ is Hermitian and nonsingular, because 
$\sigma \notin \Lambda(A)$. Moreover, $T^{1/2}S$ has the same eigenvectors $x$ as $A$.
Therefore, Lemma~\ref{lem:sin} can be applied with respect to $T^{1/2}S$, which gives
\eq{eq:lemmTS}
\sin\angle(x,v) \leq \left| \frac{\nu}{\nu_{\min}} \right| \sin\angle(x, T^{1/2}Sv),
\en 
where $\nu$ is an eigenvalue of $T^{1/2}S$ associated with the eigenvector $x$ and
$\nu_{\min}$ is the smallest magnitude eigenvalue of $T^{1/2}S$.
Similarly, by Corollary~\ref{cor:sin_subsp},
\eq{eq:corTS}
\sin\angle(x, \calQ) \leq \left| \frac{\nu_{\max}}{\nu} \right| \sin\angle(x, \calK).
\en 
Combining~\eqref{eq:lemmTS} and~\eqref{eq:corTS} with~\eqref{eq:hsaadT0}
gives the desired bound~\eqref{eq:hrr2T2}. 
\end{proof}

Clearly, the assumption that $TA = AT$ is impractical in general. Nevertheless,
the result of Theorem~\ref{thm:hsaadT} is useful in that it provides  
qualitative guidelines on the practical choice of the preconditioner $T$.
For example, it suggests an insight into ideal choices of $T$,
discussed below.  

A possible option for choosing a commuting HPD preconditioner is $T = |A-\sigma I|\inv$. 
In this case, 
$\kappa(T^{1/2}(A - \sigma I))$ in~\eqref{eq:hsaadT} turns 
into $\kappa(A - \sigma I)^{1/2}$, i.e., the condition number 
in~bound~\eqref{eq:hsaad} for the conventional harmonic Rayleigh--Ritz is
replaced by its square root.

The construction of the exact inverted absolute value is generally infeasible for 
large problems. However, in practice, one can choose $T$ as an approximation
to $|A-\sigma I|\inv$. This strategy is called the absolute value 
preconditioning~\cite{thesis, Ve.Kn:13}. Absolute value preconditioners have 
been successfully constructed and applied for computing 
interior eigenvalues of certain classes of matrices in~\cite{Ve.Kn:14TR}.

Another alternative is to set $T = (A - \sigma I)^{-2}$. In this case,
$\kappa(T^{1/2}(A - \sigma I))$ in~\eqref{eq:hsaadT} is annihilated.
Thus, in practice, a possible approach is to build~$T$ as an approximation
of $(A - \sigma I)^{-2}$. This, e.g., relates the construction of $T$ to preconditioning 
normal equations; 
see~\cite{Benzi.Tuma:03, Saad.Sosonkina:01} for a few options.  
%
Generally, however, it is hard to say which of the two preconditioning options is more 
efficient in practice; the outcomes are likely to be problem dependent. 

\section*{Acknowledgements}
The author thanks the anonymous referee, whose comments and suggestions have helped to significantly
improve this manuscript.

\bibliography{eig,plmr}

\end{document}